\documentclass[12pt]{article}
\usepackage{amsmath, amsfonts, amssymb, amsthm} 
\usepackage[colorlinks = true, linkcolor = blue]{hyperref}
\newtheorem{theorem}{Theorem}[section]

\newtheorem{definition}[theorem]{Definition}
\newtheorem{remark}[theorem]{Remark}

\newtheorem{corollary}[theorem]{Corollary}
\newtheorem{lemma}[theorem]{Lemma}

\def\Ad{\hbox{\rm Ad}}

\def\<{\,<\!}
\def\>{\!>\,}
\usepackage{graphics}
\usepackage{epsf}
\usepackage{epsfig}

\def \Ad {{\rm{Ad}}}

\def \exp {{\rm exp}}

\begin{document} 
	
	\title{Roots of elements for groups over local fields}
	\author{Parteek Kumar and Arunava Mandal}

 \maketitle

\begin{abstract}
Let $\mathbb F$ be a local field and $G$ be a linear algebraic group defined over $\mathbb F$. For $k\in\mathbb N$, let $g\to g^k$  be the $k$-th power map $P_k$ on $G(\mathbb F)$. The purpose of this article is two-fold. First, we study the power map on real algebraic group. We characterise the density of the images of the power map $P_k$ on $G(\mathbb R)$ in terms of Cartan subgroups. Next we consider the linear algebraic group $G$ over non-Archimedean local field $\mathbb F$ with any characteristic. If the residual characteristic of $\mathbb F$ is $p$, and an element admits $p^k$-th root in $G(\mathbb F)$ for each $k$, then we prove that some power of the element is unipotent. In particular, we prove that an element $g\in G(\mathbb F)$ admits roots of all orders if and only if $g$ is contained in a one-parameter subgroup in $G(\mathbb F)$. Also, we extend these results to all linear algebraic groups over global fields.

\end{abstract}


\section{Introduction}

Density of the images of the $k$-th power map on connected Lie groups has been considered in \cite{BhM}, \cite{Ma1}, and \cite{MaS}.  One of the primary motivations for these studies is to establish the relation between density of power maps and `weak exponentiality' (i.e the image of the exponential map being dense) of a connected Lie group (see \cite{BhM}), the latter property has attracted much attention in the past. In this context, we mention that there are related studies on surjectivity of the power map, and it is known that a connected real Lie group is exponential if and only if all power maps are surjective. It is also proved that the density and surjectivity of the power map is equivalent for algebraic groups over non-Archimedean local field $\mathbb Q_p$ (\cite{MR}). However, the density and surjectivity of the power maps differ on (real) Lie groups. There are various structural results obtained in this direction.

Criteria for density of the images of the power map in a connected Lie group have been given in terms of regular elements, Cartan subgroups, as well as minimal parabolic subgroups (see \cite{BhM}, \cite{Ma1}). In \cite{Ma2}, the same themes were pursued for disconnected real algebraic groups $G(\mathbb R)$ and the condition for the density of the images of the power map was given in terms of the density of the images of the power map on its connected component $G(\mathbb R)^*$ (with respect to real topology) and surjectivity in the quotient $G(\mathbb R)/G(\mathbb R)^*$, the former situation is well studied in \cite{BhM} and \cite{Ma1}. It is shown that the power map is dense on $G(\mathbb R)$ if and only if it is dense on its Levi part $L(\mathbb R)$. Therefore, for  further studies we can concentrate the (disconnected) reductive groups. This is further equivalent to the surjectivity of the power map on the set of semisimple elements of the group (see \cite{Ma2}).

Here we describe a criterion for density of the images of $P_k$ on the reductive algebraic group $G$ over $\mathbb R$, which is not necessarily Zariski connected, through the Cartan subgroup in $G(\mathbb R)$. The notion of Cartan subgroups for general (not necessarily connected) reductive group is introduced in \cite{Mo}, motivated by a definition in the case of compact Lie group given in \cite{BD}.

\begin{theorem}\label{T}
	Let $G$ be a complex algebraic group defined over $\mathbb R$, with $G^0$ is reductive. 
	Let $k\in\mathbb N$. Then the following statements are equivalent.
	\begin{enumerate}
		\item The image of $P_k:G(\mathbb R)\to G(\mathbb R)$ is dense.
		
		\item $P_k:C(\mathbb R)\to C(\mathbb R)$ is surjective for any Cartan subgroup $C$ of $G$ defined over $\mathbb R$.
		
		\item If $S(C(\mathbb R))$ is a union of Cartan subgroups in $G(\mathbb R)$, then $P_k:S(C(\mathbb R))\to S(C(\mathbb R))$ is surjective.
		
		\item If $S(G(\mathbb R))$ is the set of semisimple elements in $G(\mathbb R)$, then $P_k:S(G(\mathbb R))\to S(G(\mathbb R))$ is surjective.
	\end{enumerate}
\end{theorem}

Theorem \ref{T} can be thought as an extension of Theorem 1.1 of \cite{Ma2} and an analogue of Theorem 1.1 of \cite{BhM} in the context of real algebraic groups. In particular, this gives a method to determine density of the images of the power map, as shown by an example in \S2.
 
Now, we consider Zariski connected linear algebraic group $G$ over a non-Archimedean local field 
$\mathbb F$. In \cite{MR}, we studied the surjectivity and dense image of the power map of the linear algebraic group $G(\mathbb F)$. Moreover, we proved that if `all elements' of the algebraic group $G(\mathbb F)$ admit $k$-th root for each $k\in\mathbb N$, then $G(\mathbb F)$ is unipotent and $G(\mathbb F)$ is trivial if $\mathbb F$ is positive characteristic (\cite{MR}, \cite{Ch2}). As noted earlier there has been considerable interest in literature to understand conditions for the $k$-th power map $P_k$ to be surjective for various classes of Lie groups over local fields.
However, there does not seem to be much work on finding the conditions for existence of all the $k$-th root of a given element in a linear algebraic group over non-Archimedean local field $\mathbb F$. In \S 3, we contribute in that direction.

Note that if an element $g$ has $k$-th roots for all $k$ in a real algebraic group $G(\mathbb R)$ then $g$ might be semisimple, or  unipotent or it could be any general element. However, we proved that if $g\in G(\mathbb F)$, where $\mathbb F$ is a non-Archimedean local field, admits all the $k$-th root, then a power of $g$ is unipotent in Theorem \ref{power is unipotent}.

\begin{theorem}\label{power is unipotent}
	Let $\mathbb F$ be a non-Archimedean local field (of any characteristic) with residual characteristic $p$, and $G$ be a linear algebraic $\mathbb F$-group. Suppose that $g\in G(\mathbb F)$ has $p^k$-th roots in $G(\mathbb F)$ for all $k\in \mathbb N$. Then some power of $g$ is unipotent.
\end{theorem}

By a one parameter subgroup in a $p$-adic group $G$ we mean a continuous group homomorphism 
$\Phi:\mathbb Q_p\to G(\mathbb Q_p).$
For a connected real Lie group $G$, M. McCrudden showed that an element $g\in G$ is exponential (i.e., it is contained in a one parameter subgroup) if and only if $g$ has all $k$-th roots. Following a  work by A. Lubotzky and G. Prasad, P. Chatterjee showed that if $\Phi:\mathbb Q\to G(\mathbb Q_p)$ is an abstract group homomorphism then there exists an nilpotent element $X\in L(G(\mathbb Q_p))$ such that $\Phi(t)=\exp(tX)$ for all $t\in \mathbb Q_p$ (see Theorem 5.1 of \cite{Ch2}). In particular, $g\in G(\mathbb Q_p)$  is contained in a one parameter subgroup in $G(\mathbb Q_p)$ if and only if $g$ is unipotent. We reprove the above result in Corollary \ref{1-parameter} by using Theorem \ref{power is unipotent}.
In case, if $\mathbb F$ is of positive characteristic, and if $g\in G(\mathbb F)$ has all $k$-th roots, then we prove $g=e$ in the following Theorem \ref{rts1}. Thus Corollary \ref{1-parameter}, and Theorem \ref{rts1}
can be realised as an analogue of M. McCrudden's result (mentioned above) on algebraic groups over a non-Archimedean local field.

\begin{theorem}\label{rts1}
	Let $G$ be a linear algebraic $\mathbb F$-group over a local field $\mathbb F$ of
	characteristic $p>0$ and let $g\in G(\mathbb F)$.  Let $q$ be any prime number. Suppose for each $k$ there exists $x_k \in G(\mathbb F)$ such that
	$x_k ^{q^{k}} = g$.   Let $H$ be the Zariski-closure of the
	group generated by $g$.  Then 
	\begin{enumerate}
		\item there exists $y_k \in H(\mathbb F)$ such that $y_k ^{q^{k}}=g$.
		
		\item if the order of $g$ is finite, then the order of $g$ is co-prime to $q$.
		
		\item if $g$ has $k$-th roots in $G(\mathbb F)$ for all $k$, then $g=e$.
	\end{enumerate}
\end{theorem}
In particular, if $P_k(G(\mathbb F))=G(\mathbb F)$ for all $k$ then it follows from Theorem \ref{a fixed power}, and Theorem \ref{rts1} that $G(\mathbb F)$ is an unipotent group, and moreover $G(\mathbb F)=\{e\}$ if $\mathbb F$ is of positive characteristic (see also \cite{MR}).

In \S 4, we describe various applications of Theorem \ref{power is unipotent}, Corollary \ref{a fixed power}, and Theorem \ref{rts1} to linear algebraic groups over the global field. 

The paper is organised as follows. In \S 2, we deal with real algebraic groups and obtain Theorem \ref{T}. In  \S 3 we deal with algebraic group
over non-Archimedean local fields and deduce Theorem \ref{power is unipotent}, Theorem \ref{rts1}. We extend the results in \S 3 to the linear algebraic group over the global field in \S 4.

\section{Real algebraic group}

Now we recall the definition of a Cartan subgroup of a  complex reductive 
algebraic group $G$, which is not necessarily Zariski connected, from \cite{Mo1}. The definition is motivated from the context of disconnected compact Lie group (see Definition of \cite{BD} for details).

\begin{definition}
	An algebraic subgroup $C$ of $G$ is called a Cartan subgroup
	if the following properties hold:
	\begin{enumerate}
		\item $C$ is diagonalizable,
		
		\item $C$ has finite index in its normalizer (in $G$),
		
		\item $C/C^0$ is cyclic.
	\end{enumerate}
\end{definition}
It is noted in \cite{Mo1} that Cartan subgroup always exists in a reductive group and it is abelian. Moreover, all elements of a Cartan subgroup are semisimple.

We next note the following consequence of Proposition 2.1 of \cite{Mo}. This is easy to prove and hence we omit the details.

\begin{lemma}\label{L^*}
	Let $G$ be a complex reductive algebraic group, not necessarily Zariski-connected, defined over $\mathbb R$. Let $g\in G(\mathbb R)$ be a semisimple element. Then there exists a Cartan subgroup $C$ of $G$ defined over $\mathbb R$ such that $g\in C$.
\end{lemma}

\begin{lemma}\label{Mo}
	Let $G$ be a complex algebraic group defined over $\mathbb R$ with $G^0$ is reductive.  Let $s\in G(\mathbb R)$. Let $C$ be a Cartan subgroup of $G$ defined over $\mathbb R$ such that $sC^0$ generates $C/C^0$. Then there exists a maximal $\mathbb R$-torus $T$ (of $G^0$) containing $C^0$ such that $s\in N_{G}(T)$.
\end{lemma}

\begin{proof}
	Note that $C^0$ is a maximal torus of $Z_{G^0}(s)^0$ by \cite[Lemma 3.2]{Mo1}. Let $B'$ be a Borel subgroup of $Z_{G^0}(s)^0$ containing $C^0.$ Then by Corollary 7.4 of \cite{St} there exists a Borel subgroup $B$ in $G^0$ such that $sBs^{-1}=B$, and $B'\subset B.$ Note that $Z_{B}(C^0)$ is a maximal torus in $G^0$. Indeed, set $T=Z_{G^0}(C^0).$ As $C^0$ is a torus in $G^0$ and $G^0$ is reductive, $Z_{G^0}(C^0)$ is reductive. Hence $Z_{G^0}(C^0)=Z[Z_{G^0}(C^0),Z_{G^0}(C^0)]$, where $Z$ is the center of $Z_{G^0}(C^0)$. It is shown in Proposition 2.2 of \cite{Mo} that $[Z_{G^0}(C^0),Z_{G^0}(C^0)]$ is trivial, and hence $Z_{G^0}(C^0)=Z$ is a torus in $G^0$. As $Z_{B}(C^0)\subset Z_{G^0}(C^0)$, $Z_{B}(C^0)$ is a torus in $B$. Moreover, $Z_{B}(C^0)$ is a maximal torus in $B$ containing $C^0.$ Indeed, let $S$ be a maximal torus in $B$ containing $C^0$. Let $x\in S.$ Then $x$ commutes with each element of $C^0$, and hence $x\in Z_S(C^0)\subset Z_B(C^0)$, which implies $S\subset Z_B(C^0)$. As $S$ is maximal and $Z_B(C^0)$ is torus in $B$, $Z_{B}(C^0)$ is a maximal torus in $B$.  Since $Z_{B}(C^0)$ is a maximal torus in $B$, and $B$ is a Borel in $G^0$, $Z_{B}(C^0)$ is a maximal torus in $G^0$, and hence $Z_{B}(C^0)=Z_{G^0}(C^0)=T$.
	Since $C^0$ is defined over $\mathbb R$, $T$ is defined over $\mathbb R$. 
	Further, as $sC^0s^{-1}=C^0$, we have $sTs^{-1}=T.$ Therefore, we have a $\mathbb R$-torus $T$ of $G^0$ such that $C^0\subset T$, and $sTs^{-1}=T.$
\end{proof}

We note a reformulation of a result from \cite{Ch1}, which is useful in proving Theorem \ref{T}. To prove it we adopt a technique from Theorem 5.5 of \cite{Ch1} and Proposition 4 of \cite{St1}.

\begin{lemma}\label{P}
	Let $G$ be a complex algebraic group defined over $\mathbb R$ with $G^0$ reductive. Let $s\in G(\mathbb R)$. Let $T$ be a maximal $\mathbb R$-torus of $G^0$ such that $sTs^{-1}=T$ and $s^m\in T$. Suppose that ${\rm gcd}(m,k)=1$ and $P_k:T(\mathbb R)\to T(\mathbb R)$ is surjective. Then $P_k:Z_{T(\mathbb R)}(s)\to Z_{T(\mathbb R)}(s)$ is surjective.
\end{lemma}

\begin{proof}
	Let $x\in Z_{T(\mathbb R)}(s)$. Since $P_k(T(\mathbb R))=T(\mathbb R)$, there exists $y\in T(\mathbb R)$ such that $y^k=x$. As $sTs^{-1}=T$, we define  $\Phi_s:T(\mathbb R)\to T(\mathbb R)$ by $t\to sts^{-1}$ for any $t\in T(\mathbb R)$.
	Let us consider $\tilde x\in T(\mathbb R)$ as follows:
	$$\tilde x=y\Phi_s(y)\Phi_s^2(y)\cdots\Phi_s^{m-1}(y).$$
	Then $\Phi_s(\tilde x)=\tilde x$ and we have
	\begin{eqnarray*}
		\tilde x^k&=&(y\Phi_s(y)\Phi_s^2(y)\cdots\Phi_s^{m-1}(y))^k\\&=&(y(sys^{-1})\cdots(s^{m-1}ys^{-(m-1)}))\cdots (y(sys^{-1})\cdots(s^{m-1}ys^{-(m-1)})).\\
	\end{eqnarray*}
	Since $y$ and $\Phi_s^{i}(y)$ in $T(\mathbb R)$, and $T(\mathbb R)$ is abelian, we have
	$$\tilde x^k=y^k(sy^ks^{-1})\cdots(s^{m-1}y^ks^{-(m-1)}).$$
	Note that $\tilde x^k=(y^k)^m=x^m$ as $y^k=x\in Z_{T(\mathbb R)}(s)$.
	Since $(m,k)=1$, there exist integers $a$ and $b$ such that $am+bk=1$. This gives $x=(x^m)^a(x^b)^k=(\tilde x^ax^b)^k$. Then $\tilde y:=\tilde x^ax^b\in Z_{T(\mathbb R)}(s)$, and $\tilde y^k=x$. This proves the lemma.
\end{proof}

Now we give a proof of Theorem \ref{T}.

\noindent{\it Proof of Theorem \ref{T}}:
$(1)\Rightarrow (2)$ Suppose that the image of $P_k:G(\mathbb R)\to G(\mathbb R)$ is dense. Let $s\in C(\mathbb R)$ be such that $sC^0$ generates $C/C^0$. Our first aim is to show the following.

{\bf Claim:}
$(i)$ $s$ has $k$-th root in $C(\mathbb R)$.

$(ii)$  $P_k:C^0(\mathbb R)\to C^0(\mathbb R)$ is surjective.

{\bf Proof of the Claim:} Since $P_k (G(\mathbb R))$ is dense in $G(\mathbb R)$, by Proposition 2.6 of \cite{Ma2}, we have $P_k(G^0(\mathbb R))$ is dense in $G^0(\mathbb R)$ and $k$ is co-prime to the order of $G(\mathbb R)/G^0(\mathbb R)$. By Lemma \ref{Mo} there exists a maximal $\mathbb R$-torus $T$ (of $G^0$) containing $C^0$ such that $s\in N_{G}(T)$, where  $T=Z_{G^0}(C^0)$. Since $T$ is a maximal $\mathbb R$-torus in $G^0$ and $P_k(G^0(\mathbb R))$ is dense in $G^0(\mathbb R)$ we have $P_k:T(\mathbb R)\to T(\mathbb R)$ is surjective.
Indeed, $P_k(G^0(\mathbb R))$ is dense in $G^0(\mathbb R)$ implies $(k, G^0(\mathbb R)/G^0(\mathbb R)^*)=1$, and $P_k(G^0(\mathbb R)^*)$ is dense in $G^0(\mathbb R)^*.$ Since $T(\mathbb R)/T(\mathbb R)^*$ is a power of $2$, and $(k,2)=1$, we have $P_k:T(\mathbb R)/T(\mathbb R)^*\to T(\mathbb R)/T(\mathbb R)^*$ is surjective. As $T(\mathbb R)^*$ is connected abelian group, $P_k:T(\mathbb R)^*\to T(\mathbb R)^*$ is surjective, and hence $P_k:T(\mathbb R)\to T(\mathbb R)$ is surjective.

Let $|G(\mathbb R)/G(\mathbb R)^*|=m$. Clearly, $k$ is co-prime to $m$. Further, we see that $s^m\in G(\mathbb R)^*\subset G^0$ and $s^m$ commutes with every element of $C^0$, and hence $s^m\in Z_{G^0}(C^0)\subset T$. Therefore $P_k:Z_{T(\mathbb R)}(s)\to Z_{T(\mathbb R)}(s)$ is surjective by Lemma \ref{P}. Note that 
$Z_{T(\mathbb R)}(s)=Z_T(s)(\mathbb R)$. Hence, by Proposition 2.6 of \cite{Ma2} we have $P_k:Z_{T}(s)^0(\mathbb R)\to Z_{T}(s)^0(\mathbb R)$ is surjective. As $C^0=Z_{T}(s)^0$, the second part of the claim follows.

Let $|Z_T(s)(\mathbb R)/Z_T(s)(\mathbb R)^*|=n$. Since $P_k:Z_{T}(s)(\mathbb R)\to Z_{T}(s)(\mathbb R)$ is surjective, we get that ${\rm gcd}(k,n)=1$. As ${\rm gcd}(k,n)=1$ and ${\rm gcd}(k,m)=1$, we have ${\rm gcd}(k,mn)=1$. Recall that $s^m\in T(\mathbb R)$, and hence in $Z_T(s)(\mathbb R)$. So, $s^{mn}\in Z_T(s)(\mathbb R)^*\subset Z_T(s)^0=C^0$. Thus there exists $x\in C^0(\mathbb R)$ such that $x^k=s^{mn}$. Since ${\rm gcd}(k,mn)=1$, there exist integers $a$ and $b$ such that $ak+b(mn)=1$. So,
$$s=s^{ak}x^{bk}=(s^ax^b)^k.$$ Clearly, $s^ax^b\in C(\mathbb R)$. This proves the claim.

Finally, to prove $P_k:C(\mathbb R)\to C(\mathbb R)$ is surjective, it is enough to prove $P_k$ is surjective on $C(\mathbb R)/C^0(\mathbb R)$. This follows since $s$ is a generator of $C/C^0$, and so for $C(\mathbb R)/C^0(\mathbb R)$.

$(2)\Rightarrow (3)$ This is obvious.

$(3)\Rightarrow (4)$ Let $g\in S(G(\mathbb R)).$ Then by Lemma \ref{L^*} there exists a Cartan subgroup $C$ of $G$ defined over $\mathbb R$ such that $g\in C(\mathbb R)$.
By hypothesis, there exists an element $h\in C'(\mathbb R)$ for some Cartan subgroup $C'(\mathbb R)$ such that $h^k=g$. Since all elements of $C'(\mathbb R)$
are semisimple, the assertion follows.

$(4)\Rightarrow (1)$ Since $S(G(\mathbb R))$ is dense in $G(\mathbb R)$, the assertion follows.
\qed

The following example is an application of Theorem \ref{T}:

Let $H=S^1\ltimes \mathbb{Z}/4\mathbb{Z}$ be a compact disconnected group defined by
$$(e^{i\alpha},\bar n)(e^{i\beta},\bar m):=(e^{i\alpha}(e^{i\beta})^{(-1)^n}, \bar n+\bar m)$$ for $\alpha, \beta\in \mathbb R$, $\bar n,\bar m\in\mathbb{Z}/4\mathbb{Z}.$
We can take the subgroup $\Gamma$ generated by $(e^{i\pi}, \bar{2})$ in $H$. 

Then $\Gamma$ is a subgroup of order $2$ and also it is normal in $H$. 

We consider $G=H/\Gamma$. 
So, $G/G^*$ is of order $2$.
Then $P_k(G)$ is dense in $G$ if and only if $(k,2)=1$ by Theorem 1.1 of \cite{Ma2}. 
On the other hand, we will prove the density of $P_k(G)$ in $G$ by making use of Theorem \ref{T}. 

First we 
note that, if we take a subgroup generated by $(-1, 1)$ then it is a Cartan subgroup $C$ of $G$ and $C$ is isomorphic to $\mathbb Z/4\mathbb Z$. Therefore there are Cartan subgroups of $G$ which are isomorphic to $\mathbb Z/4\mathbb Z$. Also, note that $S^1$ and its conjugates are Cartan subgroups $G$, and in fact any Cartan subgroup (upto conjugacy) is either $S^1$ or $\mathbb Z/4\mathbb Z$ (see Exercise 5, p-181 of \cite{BD}). Note that if $(k, 2)=1$, then $P_k$ is surjective on $\mathbb Z/4\mathbb Z$. Hence by Theorem \ref{T} we have $P_k(G)$ is dense in $G$ for $(k,2)=1$.

  \section{Over non-Archimedean local field}
  An automorphism $\alpha$ of a locally compact group $G$ is called distal if $e$ is not a limit point of $\{\alpha^n(x)|~n\in \mathbb N\}$ for all $x\in G\setminus\{e\}$.
  
  It is straightforward to see that $\alpha\in {\rm GL}(V)$ is distal if and only if each eigenvalue of $\alpha$ in an extension field is of absolute value $1$.
  
  The following lemma follows from Lemma 3.1 and Proposition 3.1 of \cite{MR}. However, we include a short proof for completeness.
  \begin{lemma}\label{distal}
  	Let $\mathbb F$ be a non-Archimedean local field and $V$ be a finite dimentional vector space over $\mathbb F$. Let $p$ be the residual characteristic of $\mathbb F$.
  	Let $G\subset GL(V)$. Suppose that $g\in G$ has all $p^k$-th roots for any $k\in\mathbb N$. Then every eigenvalue of $g$ (in some finite extension of $\mathbb F$) has modulus $1$.
  \end{lemma}
  \begin{proof}
  	By Lemma 3.1 of \cite{MR}, we note that $g$ fixes a compact open subgroup $K$ of $V$. Since $\mathbb F$ is non discrete, we get a sequence $\{a_n\}$ in $\mathbb F$ such that $\{a_n\}\to 0$. As $g$ is linear, we have $g(a_nK)=a_ng(K)=a_nK$. 
  	
  	Let $x\in V$ be such that $g^n(x)\to 0$, and $a_nK$ be an open set containing $0$, we have (along subsequence if necessary) $g^n(x)\in a_nK$ for all $n$. Therefore, $x\in g^{-n}(a_nK)=a_nK$ for all $n$. Since $a_n\to 0$, we have $x=0$. This implies that $0$ is not a limit point in $V$, and hence $g$ is distal. Hence every eigenvalue of $g$ has modulus $1$.
  \end{proof}
  
  \noindent{\it Proof of Theorem~\ref{power is unipotent}}: 
  Let $p$ be the characteristic of the residue field of $\mathbb F$. It is given that for any $k\in\mathbb N$ and a prime $p$, there exists a $p^k$-th root of $g$ in $G(\mathbb F)$. Then $g$ is distal by Lemma \ref{distal}. Now $G(\mathbb F) \subset {\rm GL}(V)$, where $V$ is a finite dimensional vector space, say of dimension $n$. Then the degree of the characteristic polynomial of every element of ${\rm GL}(V)$ is $n$. 
  
  {\bf Case 1:} Now let us consider first the case where $\mathbb F$ is of characteristic $0$. Then we know that $\mathbb F$ has only finitely many extensions of any given degree (see Proposition 14, page-54 of \cite{L}). 
  Let $\mathbb K$ be the composite of all the degree $\leq n$ extensions of $\mathbb F$ contained in a fixed algebraic closure $\bar{\mathbb F}$ of $\mathbb F$. Then $\mathbb K$ is a finite extension of $\mathbb F$. 
  
  Via this construction of $\mathbb K$ we observe that eigenvalues of $g$ and its all roots belong in $\mathbb K$.
  
  Let $\mathcal O$ be its ring of integers of $\mathbb K$ and $\mathcal P$ be the unique prime ideal of $\mathcal O$.
  Since $g$ has all eigenvalues of absolute value $1$ by Lemma \ref{distal}, eigenvalues of $g$ lie in $\mathcal O$. Moreover, eigenvalues of $g$ are units in $\mathcal O$ as $g^{-1}$ also has all $p^k$-th roots, hence they lie in $\mathcal O^*$. If $\mathbb L$ is the residue field of $\mathbb K$, we know that
  $$\mathcal O^*/(1+\mathcal P)\simeq\mathbb L^*.$$
  As the residue field of $\mathbb K$ is finite, let us assume that ${\rm Ord}(\mathbb L^*)=r$.
  Note that the eigenvalues of $g^r$ belong to $1+\mathcal P$. Let $h$ be a $p^k$-th root of $g$ in $G(\mathbb F)$. As $h\in G(\mathbb F)\subset GL(V)$, we note that the eigenvalues of $h$ are in $\mathbb K$, and the absolute value of the eigenvalues of $h$ are $1$. 
  
  Moreover, the eigenvalues of $h^r$ are also contained in $1+\mathcal P$. But as $g^r = (h^r)^{p^k}$, the eigenvalues of $g^r$ are contained in $(1+\mathcal P)^{p^k}$. This is true for all integers $k$. Since, $k\in\mathbb N$ the intersection of $(1+\mathcal P)^{p^k}$ is just $\{1\}$, we get that every eigenvalue of $g^r$ is $1$. Hence $g^r$ is unipotent. 
  
  {\bf Case 2:} Now let us assume that $\mathbb F$ is of characteristic $p$. Let $\mathbb F'$ be a fixed finite field extension of $\mathbb F$ containing all the eigenvalues of $g$. We recall that for any integer $r$, there is a unique purely inseparable extension of $\mathbb F'$ of degree $p^r$. Let $r$ be chosen to be the largest integer such that $p^r\leq n$ and $\mathbb K$ be the purely inseparable extension of $\mathbb F'$ of degree $p^r$. Let $h$ be a $p^k$-th root of $g$ in $G(\mathbb F)$, and $\lambda$ be an eigenvalue of $g$. Then there exists an eigenvalue $\mu$ of $h$ such that $\mu^{p^k} = \lambda$. As $\lambda$ is contained in $\mathbb F'$, we see that $\mu$ is contained in a purely inseparable extension of $\mathbb F'$. But since the characteristic polynomial of $h$ is of degree $n$ and has coefficients in $\mathbb F$, we see that $\mu$ is contained in an inseparable extension of degree at most $n$. This implies that $\mu$ is contained in $\mathbb K$.
  
  Since residue field of $\mathbb K$ is finite, replacing $g$ by a suitable finite power, we assume that the eigenvalues of $g$ belong to $1+\mathcal P$, where $\mathcal O$ is the ring of integers of $\mathbb K$ and $\mathcal P$ is the prime ideal of $\mathcal O$.
  Therefore, the eigenvalues (of a suitable power) of $h$  is contained in $1+\mathcal P$. 
  
  Then by an argument in the preceding paragraph we see that $\lambda$ belongs to $(1+\mathcal P)^{p^k}$. This is true for all $k$. Hence $\lambda = 1$, i.e., every eigenvalue of $g$ is $1$ and so $g$ is unipotent.
  \qed
  
  Using Theorem \ref{power is unipotent} we extend Corollary 1.7 of \cite{Ch2} and Corollary 1.4 of \cite{MR}. 
  
  \begin{corollary}\label{a fixed power}
  	Let $\mathbb F$ be a non-Archimedean local field with characteristic zero and with residual characteristic $p$. Let $G$ be a linear algebraic $\mathbb F$-group. Let $E$ be a set of elements $g\in G(\mathbb F)$ which have $k$-th roots in $G(\mathbb F)$ for all $k\in\mathbb N$. Then there exists $r\in\mathbb N$ such that for each $g\in E$, $g^r$ is unipotent. In particular, if $P_k(G(\mathbb F))=G(\mathbb F)$ for all $k\in\mathbb N$, then $G(\mathbb F)$ is a unipotent group.
  \end{corollary}
  
  \begin{proof}
  	Note that, for each $n\in\mathbb N$ there exist only finitely many extensions over $\mathbb Q_p$ of degree $n$ (see Proposition 14, page-54 of \cite{L}). Hence, from the proof of Theorem \ref{power is unipotent}, it follows that there exists a fixed positive integer $r$ such that $g^r$ is unipotent for all $g\in G(\mathbb F)$. 
  	
  	Now, we observe that $$G(\mathbb F)=P_r(G(\mathbb F))=\{g^r|\;g\in G(\mathbb F)\}.$$ By above observation $P_r(G(\mathbb F))$ is unipotent, and hence $G(\mathbb F)$ is unipotent. 
  \end{proof}
  
  \begin{remark}
  	We note however that in characteristic $p>0$ case, $r$ depends on $g$.
  \end{remark}
  
  We now reprove Theorem 5.1, Corollaries 5.2 and 7.1 of \cite{Ch2}, by using Lemma 5.1 of \cite{Ch2} and Corollary \ref{a fixed power}.
  
  \begin{corollary}\label{1-parameter}
  	Let $G$ be a linear algebraic group over $\mathbb Q_p$. Let $g\in G(\mathbb Q_p).$ Then the following are equivalent:
  	
  	$(i)$ $g$ has $k$-th root in $G(\mathbb Q_p)$ for all $k$.
  	
  	$(ii)$ $g$ is unipotent $G(\mathbb Q_p)$.

  	$(iii)$ $g$ is contained in a one-parameter subgroup in $G(\mathbb Q_p)$.
  \end{corollary}
  \begin{proof}
  	Let $H=\overline{<g>}(\mathbb Q_p)$ be the Zariski closure of the subgroup generated by $g$.

  	$(i)\Rightarrow (ii)$ By Theorem \ref{power is unipotent} and Corollary \ref{a fixed power}, there exists a fixed $r\in\mathbb N$ such that $g^r$ is unipotent. Since $g$ has a $k$-th root in $G(\mathbb Q_p)$ for every $k$, by Corollary 5.1 of \cite{Ch2}, we have $$P_{p^k}:\overline{<g>}(\mathbb Q_p)\to \overline{<g>}(\mathbb Q_p)$$ is surjective for all $p$. This induces a surjective map $$P_{p^k}: \overline{<g>}(\mathbb Q_p)/\overline{<g^r>}(\mathbb Q_p)\to \overline{<g>}(\mathbb Q_p)/\overline{<g^r>}(\mathbb Q_p)$$ for all $p$.
  	Note that, $\overline{<g>}(\mathbb Q_p)/\overline{<g^r>}(\mathbb Q_p)$ is finite and the order of the group is co-prime for each prime $p$, and hence $\overline{<g>}(\mathbb Q_p)/\overline{<g^r>}(\mathbb Q_p)$ is trivial. Therefore, $g$ is unipotent.
  	
  	$(ii)\Rightarrow (iii)$ This is obvious.
  	
  	$(iii)\Rightarrow (i)$ This is obvious.
  \end{proof}
  
  Let ${\rm Ch}\mathbb F=p>0$. Let $x\in G(\mathbb F)\subset {\rm GL}(V)$ be a unipotent element. Then all its eigenvalues are $1$. So all the eigenvalues of $(x-I)$, where $I$ is the identity in ${\rm GL}(V)$, are zero. Therefore there exists $n\in\mathbb N$ such that $(x-I)^n=0$.
  So, $(x-I)^{p^n}=0$. Hence $x^{p^n}=I$. That is, $x$ is an element of finite order. Conversely, if $x^{p^n}=I$ then $(x-I)^{p^n}=0$, and hence $x$ is unipotent. 
  
  To prove Theorem \ref{rts1}, we recall a useful
  lemma from \cite{BG} (see Lemma 5.5 of \cite{BG}).
  
  \begin{lemma}\label{torsion}
  	Let $\mathbb F$ be a non-Archimedean local field of arbitrary characteristic and $m$ be a positive integer. There is an integer $r$ such that the order of every torsion element in $GL(m,\mathbb F)$ divides $r$.  
  \end{lemma}
  
  Now we give a proof Theorem \ref{rts1}.
  
  \noindent{\it Proof of Theorem~\ref{rts1}}:
  $(i)$ We first note that, since $g\in G(\mathbb F)$, and $H$ is the Zariski-closure of the group generated by $g$, $H$ is defined over $\mathbb F$. Also, note that each $x_k\in Z_G(g)$, and $Z_G(g)$ is an $\mathbb F$-group. Since $H$ is a normal subgroup of $Z_G(g)$, we have $Z_G(g)/H$ is an $\mathbb F$ algebraic group. Therefore, we may assume $(Z_G(g)/H)(\mathbb F)\subset GL(m,\mathbb F)$ for some $m$.
  
  As $x_k ^{q^{k}} = g$ for each $k$, it follows that $x_k$ is a torsion element in $(Z_G(g)/H)(\mathbb F)\subset GL(m,\mathbb F)$. By Lemma \ref{torsion}
  there exists a fixed $r\in\mathbb N$ such that the order of $x_k$ divides $r$ for each $k$. Let $m$ be the largest integer such that $q^m$ is a factor for $r.$
  Since $x_k^{q^k}=I$ in $(Z_G(g)/H)$, we must have an order of $x_k$ divides $q^m$ for each $k$. 
  Since $x_k\in G(\mathbb F)$, we have $x_k^{q^m}\in H(\mathbb F)$ for any $k$.
  
  Let $g\in G(\mathbb F)$, and $k$ be a fixed natural number. We shall show that there exists $y_k\in H(\mathbb F)$ such that $y_k^{q^k}=g.$ Set $n=m+k$. By hypothesis, there exists $x_n\in G(\mathbb F)$ such that $x_n^{q^n}=g$. 
  This implies that $(x_n^{q^m})^{q^k}=g.$ Set $y_k:=x_n^{q^m}.$ Clearly,
  $y_k\in H(\mathbb F)$, and hence this proves the desired result.

  $(ii)$ We note that, since $g$ is of finite order, $H(\mathbb F)=\left\langle g\right\rangle $ and $P_{q^k}:H(\mathbb F)\to H(\mathbb F)$ is surjective by $(i)$. We conclude that the order of $g$ is co-prime to $q$.

  $(iii)$ If $g$ has $p^k$-th roots, then there exists $r\in\mathbb N$ such that $g^r$ is unipotent by Theorem \ref{power is unipotent}. If $\alpha\in GL(V)$ is unipotent then $\alpha^{p^l}=e$ for some $l$.
  Hence $g^{rp^l}=e$. So $g$ is of finite order, and hence $H(\mathbb F)=\left\langle g\right\rangle.$ 
  By part $(ii)$, we have $$P_{q^k}:H(\mathbb F) \to H(\mathbb F)$$ is surjective for each $k$ and for each prime $q$. This implies that $g=e$.
  \qed
  
  Using Theorem \ref{power is unipotent}, we will now make two remarks.
  
  \begin{remark}
  	Let $\mathbb F$ be a non-Archimedean local field (of any characteristic) and $G$ be a Lie group over $\mathbb F$. Suppose that $\Ad: G\to GL(L(G))$ be the adjoint map, and $g\in G$ has all $k$-th roots in $G$. Then by using Theorem \ref{power is unipotent}, we have some power of $\Ad(g)\in {\rm Aut}(L(G))$ is unipotent, i.e. a power of $g$ is $\Ad$-unipotent.
  \end{remark}

  \begin{remark}
  	Let $U$ be a unipotent algebraic group over a local field $\mathbb L$ of zero characteristic, and $G$ be an algebraic group over a local field $\mathbb K$ of positive characteristic. Then the image of any abstract homomorphism $\phi: U(\mathbb L)\to G(\mathbb K)$ is trivial.
  \end{remark}
  
  \section{Algebraic groups over global field}
  
  In this section, we extend the results in \S 3 to the linear algebraic group over the global field.
  By a global field $\mathbb K$, we mean a finite extension either of $\mathbb Q$ or $\mathbb F_p(t)$.
  
  \begin{corollary}\label{a fixed power global}
  	Let $G$ be an algebraic group over a global field $\mathbb K$ with characteristic zero. Let $p$ be a prime number. Let $E$ be a set of elements $g\in G(\mathbb K)$ such that it has $p^k$-th roots in $G(\mathbb K)$ for all $k\in\mathbb N$. Then there exists $r\in\mathbb N$ such that for each $g\in E$, $g^r$ is unipotent. In particular, if $P_k(G(\mathbb K))=G(\mathbb K)$ for all $k\in\mathbb N$, then $G(\mathbb K)$ is a unipotent group.
  \end{corollary}
  
  \begin{proof}
  	Let $\mathbb K_p$ be the $p$-adic completion for $\mathbb K$. Then $G(\mathbb K)\subset G(\mathbb K_p)$ for any $p$. Let $p$ be a fixed prime. By Corollary \ref{a fixed power}, there exists $r$ (depending on $p$) such that $g^{r}$ is unipotent in $G(\mathbb K_p)$. Hence, $g^{r}$ is unipotent in $G(\mathbb K)$, as $g\in G(\mathbb K).$  
  \end{proof}
  
  We now extend Lemma \ref{torsion} on general linear group over a global field in Lemma \ref{torsion global field}, which is useful to prove Corollary \ref{rts2}.

  \begin{lemma}\label{torsion global field}
  	Let $\mathbb K$ be a global field of positive characteristic. Then there is an integer $r$ such that the order of every torsion element in $GL(m,\mathbb K)$ divides $r$.  
  \end{lemma}
  
  \begin{proof}
  	Let $\mathbb K_v$ denote a non-Archimedean local field of characteristic $p>0$ obtained by a completion of $K$. Then we have $GL(m,\mathbb K)\subset GL(m,\mathbb K_v).$ Let $x\in GL(m,\mathbb K)$ be a torsion element. Then it is torsion in $GL(m,\mathbb K_v).$ Hence there exists an integer $r$ such that ${\rm Ord}(x)$ divides $r$ by Lemma \ref{torsion}. This proves the lemma.
  \end{proof}
  
  \begin{remark}
  	Let $\mathbb K$ be a global field of positive characteristic and $G$ be a linear algebraic $\mathbb K$-group. From 
  	Lemma \ref{torsion global field}, there is an integer $r$ such that the order of every torsion element in $G(\mathbb K)$ divides $r$.  
  \end{remark}
  
  \begin{corollary}\label{rts2}
  	Let $G$ be an algebraic group over a global field $\mathbb K$ of
  	characteristic $p>0$ and let $g\in G(\mathbb K)$. Let $q$ be any prime number. Suppose for each $k$ there exists $x_k \in G(\mathbb K)$ such that
  	$x_k ^{q^{k}} = g$. Let $H$ be the Zariski-closure of the
  	subgroup generated by $g$. Then 
  	
  	$(i)$ there exists $y_k \in H(\mathbb K)$ such that $y_k ^{q^{k}}=g$.
  	
  	$(ii)$ if the order of $g$ is finite, the order of $g$ is co-prime to $q$.
  	
  	$(iii)$ if $g$ has $k$-th roots in $G(\mathbb K)$ for all $k$, then $g=e$.
  \end{corollary}
  
  \begin{proof}
  	$(i)$ Let $\Bar{\mathbb K}$ be the algebraic closure of $\mathbb K$. Then $H=\overline{<g>}^{\Bar{\mathbb K}}.$
  	Clearly, $x_k^{q^k}=I$ in 
  	$(Z_G(g)/H)(\mathbb K)\subset GL(V)$. By following the argument as in Theorem \ref{rts1}, we may conclude that there exists $y_k\in H(\mathbb K)$ such $y_k^{q^k}=g$ by Lemma \ref{torsion global field}. Hence we omit the details.
  	
  	$(ii)$ 
  	Suppose the characteristic of $\mathbb K$ is $p>0$. Let $\mathbb K_v$ denote a non-Archimedean local field of characteristic $p>0$ obtained by a completion of $K$.  Since $g\in G(\mathbb K)\subset G(\mathbb K_v)$, and it has all $k$-th roots in $G(\mathbb K)$, and hence in $G(\mathbb K_v)$, by Theorem \ref{power is unipotent}, some power of $g$ is unipotent in $G(\mathbb K_v)$. Hence a power of $g$ is unipotent in $G(\mathbb K)$. This implies $g$ is of finite order, and hence $(ii)$ holds.
  	
  	$(iii)$ Since $P_{q^k}:H(\mathbb K)\to H(\mathbb K)$ is surjective for all primes $q$ and $H(\mathbb K)$ is finite, we have the result.
  \end{proof}

\noindent
{\bf Acknowledgements:}
The second named author is grateful to G. Prasad for the email correspondence about the proof of Theorem \ref{power is unipotent}. I would also like to thank
P. Chatterjee for suggesting the statement of Theorem \ref{T}. Both Authors would like to thank C. R. E. Raja, Riddhi Shah, S. G. Dani for many useful discussion and comments. 
\medskip

\noindent
{\bf Conflict of interest:}
On behalf of all authors, the corresponding author states that there is no conflict of interest.

\vskip2mm

\begin{flushleft}
Parteek Kumar and Arunava Mandal\\
Department of Mathematics\\
Indian Institute of Technology \\
Roorkee, Uttarakhand 247667, India\\

E-mail: parteekk@ma.iitr.ac.in and  arunava@ma.iitr.ac.in

\end{flushleft}

\end{document}